\documentclass[11pt]{amsart}
\usepackage{color}

\newcommand{\bt}{\begin{Theorem}}
\newcommand{\et}{\end{Theorem}}
\newcommand{\bi}{\begin{itemize}}
\newcommand{\ei}{\end{itemize}}
\newcommand{\bea}{\begin{eqnarray}}
\newcommand{\ba}{\begin{array}}
\newcommand{\eea}{\end{eqnarray}}
\newcommand{\ea}{\end{array}}

\newtheorem{Definition}{Definition}[section]
\newtheorem{Theorem}[Definition]{Theorem}
\newtheorem{Lemma}[Definition]{Lemma}

\newtheorem*{theoremA*}{Theorem A}
\newtheorem*{theoremB*}{Theorem B}
\newtheorem*{Proofofmainthm*}{Proof of main theorem}
\newcommand{\be}{\begin{equation}}
\newcommand{\ee}{\end{equation}}
\newcommand{\bvb}{\begin{verbatim}}

\newcommand{\evb}{\end{verbatim}}

\newcommand{\wtilde}{\widetilde}%
\newcommand{\R}{\mathbb R}%
\newcommand{\C}{\mathbb C}%
\newcommand{\Z}{\mathbb Z}%
\newcommand{\N}{\mathbb N}%
\newcommand{\x}{\mathfrak X}%

\allowdisplaybreaks[1]

\renewcommand{\Re}{\mbox{Re }}
\begin{document}
\baselineskip16pt

\author[Pratyoosh Kumar and Sumit Kumar Rano]{Pratyoosh Kumar and Sumit Kumar Rano }
\address{Department of Mathematics, Indian Institute of Technology Guwahati, 781039, India.
E-mail: pratyoosh@iitg.ac.in and s.rano@iitg.ac.in}

\title[$L^p$-type estimates of Poisson transform ]
{Analysis of $L^p$-type estimates of Poisson transform on Homogeneous Trees}
\subjclass[2000]{Primary 43A85 Secondary 39A12, 20E08}
\keywords{Homogeneous Tree, Fourier Analysis, Poisson transform, Eigenfunction}

\begin{abstract}  In this article we prove the restriction theorem for Helgason-Fourier transform on homogeneous tree. Our proof is based on the duality
argument and the norm estimates of Poisson transform. We also characterize all eigenfunctions of the laplacian
on homogeneous tree which are Poisson transform of $L^p$ functions defined on the boundary.
\end{abstract}

\maketitle

\section{Inroduction}
The formulation of Fourier restriction theorem on $\R^n$ ($n\geq 2$) emerges explicitly by the work of Stein. It says that $S^{n-1}$ the unit sphere
in $\R^n$ ($n\geq 2$) satisfies a $(p,q)$  \emph{restriction theorem} if
$$\left(\int_{S^{n-1}} |\mathcal{F}(f)(\xi)|^qd\sigma(\xi)\right)^{1/q}\leq C_{(p,q,n)}\|f\|_{L^{p}(\R^n)}$$
holds for each $f \in L^1\cap L^p$, where $\mathcal{F}(f)$ is the Fourier transform of $f$.

One of the celebrated result in this context is the Tomas-Stein restriction theorem. It says that the Fourier transform $\mathcal{F}(f)$ of a function
$f\in L^p(\R^n)$ has a well defined restriction on the unit sphere $S^{n-1}$ via the inequality,
$$\|\mathcal{F}(f)|_{S^{n-1}}\|_{L^{2}(S^{n-1})}\leq C_{p,n}\|f\|_{L^{p}(\R^n)}\quad\text{for all}~1\leq p\leq\frac{2n+2}{n+3}.$$

The similar results are also known for hyperbolic space, more generally rank one symmetric space of non compact type and its non symmetric
generalization Damek--Ricci spaces (see \cite{K,KRS,LR}).
In this article we prove the similar restriction theorem for Helgason-Fourier transform on homogeneous tree. The homogeneous tree may be considered as a discrete model of hyperbolic space.
 Many authors have extended the analogous results of harmonic analysis on this structure (see e.g. \cite{CS,CMS,CMS1,FTC,FTP2,SP,V}).

\subsection{Motivation and Statement of main result}
 Let $\x$ be a homogeneous tree of degree $q+1$ and $o$ be some fixed arbitrary reference point in $\x$. The boundary  of $\x$ is the set of all infinite geodesic rays starting at $o$ and is denoted by $\Omega$.  The Helgason-Fourier transform $\wtilde{f}$ of a finitely supported function $f$ on $\x$ is a function on $\mathbb C\times\Omega$ defined by the formula
\[\wtilde{f}(z,\omega)=\sum\limits_{x\in\x}f(x)p^{1/2+iz}(x,\omega)\]
where $p(x,\omega)$ is the Poisson kernel. For details of notations and definitions, we refer Section 2 and the references given there.


For fixed $z\in\C$, we say that  $\Omega$ the boundary of $\x$ satisfies a $(p,q)$  \emph{restriction theorem} if
\be \left(\int_{\Omega} |\wtilde{f}(z,\omega)|^qd\nu(\omega)\right)^{1/q}\leq C_{(p,q)}\|f\|_{L^{p}(\x)}\ee
holds for each $f \in L^p(\x)$. Before the formulation of main result for general functions let us first discuss the special case of radial function.
A simple calculation shows that if $f$ is a radial function then
\be \widetilde{f}(z,\omega)=\hat{f}(z)=\sum\limits_{x\in\mathfrak{X}}f(x)\varphi_{z}(x)\label{eq2}\ee where $\varphi_{z}$ is the spherical function.
For $p\in(1,\infty)$ define
$$\delta_p=\frac{1}{p}-\frac{1}{2}\;\;\; \text{and}\;\;\; S_p=\{z\in\C: |\Im z|\leq|\delta_p|\}.$$ Let $S_p^\circ,~\partial{S_p}$  be the usual interior and the boundary of $S_p$ respectively. Note that $\delta_p=-\delta_{p^\prime}$ and  $S_2=\R$. We define $\delta_1=-\delta_\infty=1/2$ and $S_1=\{z\in\C: |\Im z|\leq1/2\}.$

It is well known that $\varphi_{z}$ is bounded if and only if $z\in S_1.$ For other value of $p$ and for any $x\in\x$, we have the following pointwise estimates of $\varphi_{z}(x)$:
\begin{enumerate}
\item  For $1<p<2$, $|\varphi_{z}(x)|\asymp q^{-\frac{|x|}{p'}}$ if $\Im z=\delta_{p^{\prime}}$
\item For $p=2$, $|\varphi_{z}(x)|\asymp q^{-|x|/2}(1+|x|)$ if $z\in(\tau/2)\Z$ and
 $|\varphi_{z}(x)|\leq C_{z}q^{-|x|/2}$ if $z\in\R\setminus(\tau/2)\Z,$ where $\tau=\frac{2\pi}{\log q}.$
\end{enumerate}

The following theorem is a consequence of the above estimates.

\begin{Theorem}\label{sfestimate}
Let $1<p<2$. Then
\begin{enumerate}
\item[$1.$] $\varphi_z\in L^{p'}(\x)\quad\text{if and only if}\quad z\in S_p^\circ$
\item[$2.$] $\varphi_z\in L^{p',\infty}(\x)\quad\text{if and only if}\quad z\in S_p$
\item[$3.$] $\varphi_z\notin L^{2,\infty}(\x)$ if $z\in(\tau/2)\Z$ and $\varphi_z\in L^{2,\infty}(\x)$ if $z\in\R\setminus(\tau/2)\Z$.
\end{enumerate}
\end{Theorem}

From equation (\ref{eq2}) and the estimates of $\varphi_z$, we have the following.
\begin{Theorem}
Let $1\leq p<2$. Then
\begin{enumerate}
\item[$1.$] $\hat{f}(z)$ exists if $f\in {L^{p}(\mathfrak{X})}^\sharp$ and $z\in S_p^\circ$
\item[$2.$] $\hat{f}(z)$ exists if $f\in {L^{p,1}(\mathfrak{X})}^\sharp$ and $z\in S_p$
\item[$3.$] $\hat{f}(z)$ exists if $f\in {L^{2,1}(\mathfrak{X})}^\sharp$ and $z\in\R\setminus(\tau/2)\Z$.
\end{enumerate}
\label{radial restriction}\end{Theorem}
The above estimates of spherical function are well known and can be found in the literature (see e.g.\cite{CMS1,CS} and the reference given there in).
Our goal is to extend Theorem \ref{radial restriction} to the general case. Now we state the restriction theorems for Helgason-Fourier transform on homogeneous tree.
These theorems  can also be considered as an analogue of restriction theorems proved in \cite{K} and \cite{KRS}.
\begin{theoremA*}Let $1\leq p<2.$
\begin{enumerate}
\item Suppose $f\in L^{p}(\mathfrak{X})$. For $p<r<p'$ and $z\in\mathbb{C}$ with $\Im z=\delta_{r^{\prime}},$ there exists a constant $C_{p,r}>0$ such that
$$\left(\int\limits_{\Omega}|\widetilde{f}(z,\omega)|^r d\nu(\omega)\right)^{1/r}\leq C_{p,r}\|f\|_{L^{p}(\mathfrak{X})}.$$
If $p=1$ then $r\in [1,\infty ]$ and $C_{p,r}=1$.
\item Suppose $f\in L^{p,1}(\mathfrak{X})$. For $r=p~\text{or}~p'$ and $z\in\mathbb{C}$ with $\Im z=\delta_{r^{\prime}},$ there exists a constant $C_{p}>0$ such that
$$\left(\int\limits_{\Omega}|\widetilde{f}(z,\omega)|^r d\nu(\omega)\right)^{1/r}\leq C_{p}\|f\|_{L^{p,1}(\mathfrak{X})}.$$
\end{enumerate}
\end{theoremA*}
Extension of the above result for the case $p=2$ is the following theorem.
\begin{theoremB*}
Let $f\in L^{2,1}(\mathfrak{X})$. If $z\in\mathbb{R}\setminus(\tau/2)\Z,$ then there exists a constant $C_{z}>0$ such that
$$\left(\int\limits_{\Omega}|\widetilde{f}(z,\omega)|^2d\nu(\omega)\right)^{1/2}\leq C_{z}\|f\|_{L^{2,1}(\mathfrak{X})}.$$
\end{theoremB*}
 We will prove the above theorem  using duality argument. This is based on the fact that the norm estimates of Helgason--Fourier
 transform on the boundary are equivalently related to the norm estimates of Poisson transform. It is also  worth mentioning that Poisson transform of $L^p(\Omega)$
 functions are eigenfunctions of the Laplacian. It was proved by Mantero and Zappa in
  \cite{MZ} that all eigenfunctions of Laplacian are given by  Poisson transform of martingale on the boundary.
  In this paper we characterize all such eigenfunctions of the Laplacian which are Poisson transform of functions on $L^p(\Omega)$.

Since Poisson transforms are certain matrix coefficients of representations on  $G$ (the group of isometries of $\x$)
therefore its estimates play very important role in harmonic
  analysis on homogeneous tree and its group of isometries.  In fact using the estimates (\ref{pcase}) and (\ref{p'case}) Cowling et. al. in \cite{CMS} proved
  the following version of generalized Kunze--Stein phenomenon
  \be L^{p,r}(G)\ast L^{p,s}(G)\subseteq L^{p,t}(G),\label{cowling}\ee
where $1<p<2$, $\frac1r+\frac1s\geq\frac1t$ and $r,s,t \in [1,\infty]$.
In \cite{V} Veca extended the above convolution relation and proved the end point version of (\ref{cowling})
\be L^{2,1}(G)\ast L^{2,1}(G)\subseteq L^{2,\infty}(G).\label{endpoint}\ee

 Regarding the estimates of the Poisson transform, the case $p=2$ is slightly different.
   This can be seen in the behavior of the elementary spherical
 function $\varphi_z$ (which are Poisson transform of constant function $1$). In fact it is mentioned in Theorem \ref{sfestimate}
 that $\varphi_{\alpha\pm i\delta{p}}\in L^{p^{\prime},\infty}(\x)$ for all $\alpha\in\R.$ However
 $\varphi_z\notin L^{2,\infty}(\x)$ if $z\in(\tau/2)\Z$ and $\varphi_z\in L^{2,\infty}(\x)$ if $z\in\R\setminus(\tau/2)\Z.$ In this article
we also characterize all such eigenfunctions of the Laplacian which are Poisson transform of functions on $L^2(\Omega)$ corresponding to the parameters $z\in\R\setminus(\tau/2)\Z.$


 The paper is structured in three small sections apart from the introduction. Section 2 provides the notation, definition and basic
  known results that we will use in this paper. In Section 3 we discuss the characterization of $L^p$-type eigenfunction of Laplacian. The restriction theorems are proved in Section 4.

\section{Notation and Preliminary Results}
Most of our notation and results are standard. We will mainly follow the notation of \cite{CS,CMS,FTC,FTP1,ST}. The letters $\C$ and
 $\R$ will denote the set of complex and real numbers respectively. For $z\in \C$ we use the notation $\Re z$ and $\Im z$ for real and imaginary part of
  $z$ respectively. Everywhere in this paper, any $p\in(1,\infty)$ is related to $p^\prime$ by the relation $\frac{1}{p}+\frac{1}{p^\prime}=1$.
  We will use the standard practice of using the letter $C$ for constant, whose  value may change from one line to another line. Occasionally
  the constant $C$ will be suffixed to show its dependency on related parameters. Given positive functions $A$ and $B$ defined on a set $X$, we say that $A\asymp B$
in $X$ if there exists positive constants $C_1$ and $C_2$ such that $C_1A(t)\leq B(t) \leq C_2A(t)\;\; \forall t\; \in X$.

\subsection{Homogeneous tree }Let $\x$ be a homogeneous tree of degree $q+1\;(q>1)$ that is, a connected graph with no loops, in which every vertex is adjacent to $q+1$ other vertices.
Naturally, the measure on $\x$ is the counting measure. For any finite subset $E$ of $\x,$ let $\#E$ denotes the number of vertices in $E$.
The distance $d(x,y)$ between two vertices $x$ and $y$ is defined as the number of edges joining $x$ and $y$. We write $S(x,n)=\{y\in \x: d(x,y)=n\}$ and
$B(x,n)=\{y\in \x: d(x,y)\leq n\}$. Clearly $$\#S(x,n)=(q+1)q^{n-1}\asymp q^n\:\; \text{and}\;\;  \#S(x,n)\asymp\#B(x,n).$$
 Let $o$ be a fixed reference point in $\x$ and denote by $|x|$ the distance of $x$ from $o$. Let $G$ be the group of isometries of the metric space $(\x,d)$ and
 let $K$ be the stabilizer of $o$ in $G$. The map $g\rightarrow g\cdot o$ identifies $\x$ with the coset space $G/K$, so that functions on $\x$ corresponds to $K$-right
  invariant functions on $G$.  Further radial functions on $\x$, that is, functions which only depend on $|x|$, corresponds to $K$-bi-invariant functions on $G$. If $E(\x)$ is a
   function space on $\x$  we will denote by $E(\x)^\#$  the radial functions in $E(\x)$.

\subsection{Boundary of $\x$} An infinite geodesic ray $\omega$ in $\x$ is an one-sided sequence $\{\omega_{n}: n=0,1,2\ldots \}$ where $\omega_n$'s are
in $\x$. These infinite geodesic rays are identified if there exists non-negative integers $i$ and $j$ such that $\omega_{n}=\omega^{\prime}_{n+i}$ for all $n\geq j$.
This identification is an equivalence relation. Let [$\omega$] denote the equivalence class of $\omega$. In every equivalence class [$\omega$],
there is a unique geodesic ray, denoted by $\omega$, starting at $o$. The boundary of $\x$ is the set of all infinite geodesic rays starting at $o$,
and will be denoted by $\Omega$. Notice that the map $k\rightarrow k\cdot\omega$ represents a transitive action of $K$ on $\Omega$.

If $x$ is in $\x$ and $\omega$ is in $\Omega$, we define $c(x,\omega)=x_l$ where $x_l$ is the last point lying on $\omega$ in the geodesic path $\{o,x_1,\ldots,x\}$ joining $o$ to $x$.
  For $x\in \x$ we define $E_j(x)=\{\omega\in \Omega : |c(x,\omega)|\geq j\}$ for all $j\geq0$.  Note that $E_0(x)=\Omega$ and $E_j(x)=\emptyset$ for $j>|x|$ and
  denote $E_{|x|}(x)$ by $E(x)$. The sets $E_j(x)$ are open subsets of $\Omega$ and indeed forms a basis. There exists a unique $K$-invariant probability measure
   $\nu$ such that $$\nu(E_j(x))=\dfrac{q}{(q+1)q^j}.$$
  Thus $(\Omega,\mathcal{M},\nu)$ is a measure space where $\mathcal M$ is a $\sigma$ algebra generated by the sets $\{E(x):\; x\in \x\}$. Let $\mathcal M_{n}$ be
   the $\sigma$ sub-algebra of $\mathcal M$ generated by $E(x),~|x|\leq n.$ The conditional expectation of a locally integrable function $F$ on $(\Omega, \mathcal{M}, \nu)$
   relative to the $\sigma$ sub-algebra $\mathcal{M}_{n},~n\geq0$ is given by
\be \mathcal{E}_{n}(F)(\omega)=\frac{1}{\nu(E_{n}(\omega))}\int\limits_{E_{n}(\omega)}F(\omega')d\nu(\omega')\label{maximal}\ee
and the $n$th difference operator is given as $\Delta_{n}(F)=\mathcal{E}_{n}(F)-\mathcal{E}_{n-1}(F)$ where $n\geq 0$ and $\mathcal{E}_{-1}=0$.
 Notice that $\mathcal{E}_{n}(F)=\sum\limits_{j=0}^{n}\Delta_{j}(F)$ whenever $n\geq 1$ and $\mathcal{E}_{0}(F)=\Delta_{0}(F)$.
  The conditional expectation and the difference operator satisfies $\mathcal{E}_{m}(\mathcal{E}_{n}(F))=\mathcal{E}_{k}(F)$ where $k=\min\{m,n\}$ and
\be\langle \Delta_{m}(F),\Delta_{n}(F)\rangle_{L^{2}(\Omega)}=
 \begin{cases}
0 &\text{when }m\neq n\\
\|\Delta_{n}(F)\|^{2}_{L^{2}(\Omega)}&\text{when }m=n\\
\end{cases}
\label{ortho} \ee
respectively. Further the maximal operator defined by the formula $\mathcal{E}(F)(\omega)=\sup\limits_{n\geq 0}|\mathcal{E}_{n}(F)(\omega)|$, is weak type $(1,1)$ and strong type $(p,p)$ whenever $p>1$. We refer \cite[Ch.IV]{ST} for more details.

\subsection{Poisson transformation and Eigenfunction of $\mathcal L$} On the boundary $\Omega,$ $\nu$ is the $G$-quasi-invariant probability measure and the Poisson kernel $p(g\cdot o,\omega)$ is defined to be the Radon-Nikodym derivative $d\nu(g^{-1}\omega)/d\nu(\omega)$.
The height $h_{\omega}(x)$ of $x$ in $\x$ with respect to $\omega$ is defined by the formula \[h_{\omega}(x)=2|c(x,\omega)|-|x|.\]
 Then the Poisson kernel is given by
$$ p(x,\omega)=q^{h_{\omega}(x)} \;\; \forall x\in \x \;\; \forall \omega\in\Omega.$$
So the Poisson kernel is a function on $\x\times \Omega$ and can also be written as
 \be p(x,\omega)=\sum\limits_{j=0}^{|x|}q^{2j-|x|}\mathcal{X}_{E_j(x)\setminus E_{j+1}(x)}(\omega)\;\; \forall x\in \x \;\; \forall \omega\in\Omega.\label{poissonfull}\ee

For $z\in\C$, we define the representations $\pi_{z}$ of $G$ on $C(\Omega)$ by the formula
$$\pi_{z}(g)\eta(\omega)=p^{1/2+iz}(g\cdot o,\omega)\eta(g^{-1}\omega)\quad\forall g\in G~~ \forall\omega\in\Omega.$$
It is obvious that $\pi_{z}=\pi_{z+\tau},$ where $\tau=2\pi/\log q$. We denote the torus $\R/\tau \Z$ by $\mathbb{T}$, which can be identified with the interval $[-\tau/2,\tau/2)$.
The Poisson transformation $\mathcal{P}_{z}:C(\Omega)\rightarrow C(\x)$ is given by the formula
 $$\mathcal{P}_{z}\eta(x)=\left\langle \pi_{z}(x)1,\eta\right\rangle=\int\limits_{\Omega}p^{1/2+iz}(x,\omega)\eta(\omega)d\nu(\omega).$$

  The Laplace operator (or Laplacian) $\mathcal L$ on $\x$ is defined by $$\mathcal Lf(x)=\frac{1}{q+1}\sum\limits_{y:d(x,y)=1}f(y).$$
  Now we give a brief summary of eigenfunctions of the
  Laplacian. We refer \cite[Ch.II]{FTC} for details.  For fixed $\omega$, the function $x\to p^{1/2+iz}(x,\omega)$ is an
  eigenfunction of the Laplace operator with eigenvalue $\gamma(z)$ and therefore
  $$\mathcal L \mathcal{P}_{z}\eta(x) =\gamma(z) \mathcal{P}_{z}\eta(x)\;\;  \text{where} \;\; \gamma(z)=\frac{q^{1/2+iz}+q^{1/2-iz}}{q+1}.$$

Let $\mathcal{K}_{n}(\Omega)$ be a linear space of functions on $\Omega$ which are linear combinations of characteristic functions of the sets
 $E(x),~|x|\leq n$ and define $\mathcal{K}(\Omega)=\bigcup\limits_{n\geq 0}\mathcal{K}_{n}(\Omega)$. A martingale $\mathbf F=(F_{n})_{n\geq 0}$ is
 such that each $F_{n}\in \mathcal{K}_{n}(\Omega)$ and $\mathcal{E}_{m}(F_{n})=F_{k}$ where $k=\min\{m,n\}$. It is easy to see that if $\mathbf F=(F_{n})$ is a
  martingale then for each $n\geq 0$, $F_{n}=\sum\limits_{j=0}^{n}\Delta_{j}(F_{m})$ whenever $m\geq n$. It is also interesting to observe that every $F$
  in $L^{p}(\Omega)$, where $1\leq p<\infty$ can be identified to the martingale $\mathbf{F}=(\mathcal{E}_{n}(F))$ via the
  conditional expectation, with $\mathcal{E}_{n}(F)$ tending to $F$ in $L^{p}$ norm as $n\rightarrow\infty$. However, not every martingale corresponds to a $L^{p}$ function.
  But if a martingale $\mathbf{F}=(F_{n})$ satisfies $\sup\limits_{n\geq 0}\|F_{n}\|_{L^{p}(\Omega)}<\infty$ for some $p>1$, then there exists an $F$ in $L^{p}(\Omega)$
  such that $\mathcal{E}_{n}(F)=F_{n}$ for all $n\geq 0$.

The dual space $\mathcal{K}'(\Omega)$ of $\mathcal{K}(\Omega)$ identifies to the space of all martingales which means that every linear functional $F$ defined on $\mathcal{K}(\Omega)$ corresponds to a unique martingale $\mathbf F=(F_{n})$ and is given by
$$F(\eta)=\lim\limits_{n\rightarrow\infty}\int\limits_{\Omega}F_{n}(\omega)\eta(\omega)d\nu(\omega)\quad \forall\eta\in\mathcal{K}(\Omega).$$
Using the duality above, we can extend the definition of the Poisson transformation to a martingale $\mathbf F=(F_{n})$ as
$$\mathcal{P}_{z}{\mathbf F}(x)=\lim\limits_{n\rightarrow\infty}\int\limits_{\Omega}p^{1/2+iz}(x,\omega)F_{n}(\omega)d\nu(\omega)
.$$
It follows from the
definitions that the  Poisson transform of a martingale is an eigenfunction of the Laplacian that is if  $\mathbf F$ is a martingale then $\mathcal{P}_{z}{\mathbf F}\in \mathbb{E}_{z}(\x),$
where  $$\mathbb{E}_{z}(\x)=\{u\in C(\x): \mathcal{L}u(x)=\gamma(z)u(x) \;\;\forall x\in\x\}$$
is the eigen-space of the Laplace operator with eigenvalue $\gamma(z)$.
The following theorem proved in \cite[Ch.II, Theorem 1.2]{FTC}  gives the complete characterization of eigenfunctions of $\mathcal L$ in terms of Poisson transform of martingale (see also \cite[Corollary 3.5]{MZ}).
 \begin{Theorem} Let $z\in\C$. If $z\neq (k\tau+i)/2$ where $k\in\Z$ then the map $\mathcal P_z: \mathcal{K}'(\Omega)\to\mathbb{E}_z(\x)$ is a bijection.\label{bijection}
 \end{Theorem}

The elementary spherical function $\varphi_z$ is the radial eigenfunction of the Laplacian normalized by $\varphi_z(0)=1$. It can also be represented by
$$\varphi_z(x)=\mathcal{P}_z 1=\int\limits_{\Omega}p^{1/2+iz}(x,\omega)d\nu(\omega).$$

It is known that
\begin{equation}\label{eqsf}
\varphi_z(x)= \begin{cases}
\vspace*{.2cm} \left(\frac{q-1}{q+1}|x|+1\right)q^{-|x|/2}&\forall z\in\ \tau\Z\\
\vspace*{.2cm}\left(\frac{q-1}{q+1}|x|+1\right)q^{-|x|/2}(-1)^{|x|}&\forall z\in {\tau/2}+\tau\Z\\
 \mathbf{c}(z)q^{{(iz-1/2)}|x|}+\mathbf{c}(-z)q^{{(-iz-1/2)}|x|}&\forall z\in\C\setminus(\tau/2)\Z,
\end{cases}
\end{equation}
where $\mathbf{c}$ is a meromorphic function given by \[\mathbf{c}(z)=\frac{q^{1/2}}{q+1}\frac{q^{1/2+iz}-q^{-{1/2}-iz}}{q^{iz}-q^{-iz}}\quad \forall z\in\C\setminus(\tau/2)\Z.\]
It is easy to see that $|\varphi_{z}(x)|\leq 1$ for all $x\in\x$ whenever $z\in S_1$. Note that if $1\leq r\leq \infty$ and $z=\alpha+i\delta_{r^{\prime}},~\alpha\in\R$
then by H\"{o}lder's inequality we have
\be\|\mathcal{P}_{z}\eta\|_{L^\infty(\x)}=\|\left\langle \pi_{z}(\cdot)1,\eta\right\rangle\|_{L^\infty(\x)}\leq \|\eta\|_{L^{r^\prime}(\Omega)}\label{bounded poisson}\ee
for all $\eta \in C(\Omega)$. In particular if $\eta \in C(\Omega)$ then $\mathcal{P}_{z}\eta\in L^\infty(\x)$.
\subsection{The Helgason--Fourier Transform on $\x$}
Let $\mathcal{D}(\x)$ be the space of all finitely supported functions on $\x$. The Helgason-Fourier transform $\wtilde{f}$ of a function $f$ in $\mathcal{D}(\x)$ is a function on $\mathbb \C\times\Omega$ defined by the formula
\[\wtilde{f}(z,\omega)=\sum\limits_{x\in\x}f(x)p^{1/2+iz}(x,\omega).\]

The spherical Fourier transform $\hat{f}$ of a function  $f\in D(\x)^\#$  is defined by
$$\hat{f}(z)=\sum\limits_{x\in\x}f(x)\varphi_z(x)$$
for all $z\in\C$. It can be easily shown and a well known fact that the spherical Fourier transform of a radial function coincides with the full Helgason-Fourier transform. That is if $f \in D(\x)^\#$ then
$$\widetilde{f}(z,\omega)=\hat{f}(z)=\sum\limits_{x\in\mathfrak{X}}f(x)\varphi_{z}(x).$$
The inversion formula  for Helgason-Fourier transform of $f\in D(\x)$ (see \cite[Theorem 2.6]{ CMS}) is given by
$$f(x)=\int_{\mathbb T}\int_{\Omega}\widetilde{f}(s,\omega)p^{1/2-is}(x,\omega)d\nu(\omega) d\mu(s).$$
In particular if $f\in D(\x)^\#$, then
$$f(x)=\int_{\mathbb T}\hat{f}(s)\varphi_{-s}(x) d\mu(s).$$
If $f_1$ and $f_2$ are in $D(\x),$ then
$$\sum\limits_{x\in\x}f_1(x)\overline{f_2}(x)=\int_{\mathbb T}\int_{\Omega}\widetilde{f_1}(s,\omega)\overline{\widetilde{f_2}(s,\omega)}d\mu(s)d\nu(\omega).$$

The Helgason--Fourier transformation extends to an isometric map from $L^2(\x)$ into $L^2(\mathbb T\times\Omega, d\mu(s)d\nu(\omega))$. In fact its range is the subspace of
$L^2(\mathbb T\times\Omega, d\mu(s)d\nu(\omega))$ of the functions $g$ which satisfy the following symmetry condition
$$\int_{\Omega}p^{1/2-is}(x,\omega) g(s,\omega)d\nu(\omega)=\int_{\Omega}p^{1/2+is}(x,\omega) g(-s,\omega)d\nu(\omega).$$
for every $x\in\x$ and and almost every $s$ in $\mathbb T$. Here $\mu$ denotes the Plancherel measure whose density with respect Lebesgue maesure is given by $C_G|\mathbf c(s)|^{-2}$.
\subsection{Basic properties of Lorentz spaces}
We also need some basic facts about the Lorentz spaces ( for details see \cite{LG1}). Let $(M,m)$ be a $\sigma$-finite measure space, $f:M\rightarrow\C$ be a measurable function. The distribution function $d_{f}:(0,\infty)\rightarrow (0,\infty]$ and the nonincreasing rearrangement $f^{*}:(0,\infty)\rightarrow (0,\infty]$ are defined by the formulae \[d_{f}(s)=m\big(\{x\in M: |f(x)|>s\}\big)\quad\text{and}\quad f^{*}(t)=\inf\{s:d_{f}(s)\leq t\}\]
For $p\in[1,\infty)$ and $q\in[0,\infty]$, we define
\[\|f\|_{p,q}=\begin{cases}
\Big(\frac{q}{p}\int\limits_{0}^{\infty}[f^{*}(t)t^{1/p}]^q \frac{dt}{t}\Big)^{1/q}&\text{when}~q<\infty,\\
\sup\limits_{t>0}t^{1/p}f^{*}(t)=\sup\limits_{t>0}td_{f}(t)^{1/p}&\text{when }q=\infty.
\end{cases}\]
For $p\in[1,\infty)$ and $q\in[0,\infty)$, we define the {\it{Lorentz~space}} $L^{p,q}(M)$ as follows:
\[L^{p,q}(M)=\{f:M\rightarrow\C:f~\text{measurable and}~\|f\|_{p,q}<\infty\}.\]
By $L^{\infty,\infty}(M)$ and $\|\cdot\|_{\infty,\infty}$ we mean the space $L^{\infty}(M)$ and the norm $\|\cdot\|_{\infty}$ respectively and for the other values of $p$ we have $L^{p,p}(M)=L^{p}(M)$. For $1<p<\infty$ the space $L^{p,\infty}(M)$ is known as the weak $L^p$ space and also $L^{p,q}\subset L^{p,s}$ for all $1\leq q\leq s\leq\infty$.

\section{characterization of eigenfunctions}
In this section, we shall characterize the eigenfunctions of the Laplacian $\mathcal L$ which are Poisson transform of $L^p$ functions defined on the boundary $\Omega$. We begin with the following lemma:
\begin{Lemma}\label{lemmaweak}
If $1<p\leq 2,$ then for every $u\in L^{p^{\prime},\infty}(\x)$ there exists a positive constant $C$  (independent of $u$) such that
\begin{equation}
\frac{1}{N}\sum\limits_{x\in B(0,N)}|u(x)|^{p^{\prime}}\leq C\|u\|^{p^{\prime}}_{L^{p^{\prime},\infty}} \;\;\; \forall N\in\mathbb{N}.
\label{eq7}\end{equation}
\end{Lemma}
We leave the proof of the above lemma. It is easy and follows from the properties of Lorentz spaces.
\begin{Theorem} Let $1<p<2$. Suppose that  $u\in C(\x)$ and $z=\alpha+i\delta_{p^{\prime}},~\alpha\in\mathbb{R}~$. Then $u(x)=\mathcal{P}_{z}F(x)$ for
some $F\in L^{p^{\prime}}(\Omega)$ if and only if $u\in L^{p^{\prime},\infty}(\x)$ and  $\mathcal L u(x)=\gamma(z) u(x)$ .
 Moreover there exists positive constants $C_{1},~C_{2}$ such that for all $F\in L^{p^{\prime}}(\Omega)$ we have
\be C_{1}\|\mathcal{P}_{z}F\|_{L^{p^{\prime},\infty}(\mathfrak{X})}\leq\|F\|_{L^{p^{\prime}}(\Omega)}\leq C_{2}\|\mathcal{P}_{z}F\|_{L^{p^{\prime},\infty}(\mathfrak{X})}\label{pcasemain}.\ee
\end{Theorem}

\begin{proof}
We first prove that if $z=\alpha+i\delta_{p^{\prime}},~\alpha\in\R$ ($1<p<2$) then for all $F\in L^{p^{\prime}}(\Omega),~\mathcal{P}_{z}F\in L^{p^{\prime},\infty}(\x)$ and
\be
\|\mathcal{P}_{z}F\|_{L^{p^{\prime},\infty}(\x)}\leq C_{p}\|F\|_{L^{p^{\prime}}(\Omega)}.\label{pcase}
\ee
The estimate (\ref{pcase}) is already available in the literature (see \cite{CMS}), however for the sake of completeness, we give the sketch of the proof.
Let $x\in\mathfrak{X}$ and $\lbrace{o=x_0,x_1, x_2,\ldots,x_n=x\rbrace}$ be the geodesic connecting $o$ to $x$. Then for $z=\alpha+i\delta_{p^{\prime}},~\alpha\in\R$  and $F\in L^{p^{\prime}}(\Omega),$ using (\ref{poissonfull}) we have
\begin{align*}
|\mathcal{P}_{z}F(x)|
&=\left | q^{-|x|(1/p+i\alpha)}\sum\limits_{j=0}^{|x|}q^{2j(1/p+i\alpha)}\int\limits_{E_{j}(x)\setminus E_{j+1}(x)}F(\omega)d\nu(\omega)\right |\\
&\leq q^{-|x|/p}\sum\limits_{j=0}^{|x|}q^{2j/p}\int\limits_{E_{j}(x)}|F(\omega)|d\nu(\omega)\\
&=q^{-|x|/p}\left[\frac{q}{q+1}\sum\limits_{j=1}^{|x|}q^{(2/p-1)j}\mathcal{E}_{j}(|F|)(\omega)+\mathcal{E}_{0}(|F|)(\omega)\right]\\
&\leq \mathcal{E}(|F|)(\omega)q^{-|x|/p}\sum\limits_{j=0}^{|x|}q^{(2/p-1)j} =\mathcal{E}(|F|)(\omega)q^{-|x|/p}\frac{q^{(2/p-1)(|x|+1)}-1}{q^{2/p-1}-1}\\
&\leq C_{p} q^{-|x|/p^{\prime}}\mathcal{E}(|F|)(\omega)\quad\forall x\in\mathfrak{X},~\omega\in E(x).
\end{align*}
For $\lambda>0$, define the set $E_{\lambda}=\{x\in\mathfrak{X}:|\mathcal{P}_{z}F(x)|>\lambda\}.$ It is easy to show that
$$\#E_{\lambda} \leq C_{p}\frac{{\|\mathcal{E}(|F|)\|}^{p^\prime}_{L^{p^{\prime}}(\Omega)}}{\lambda^{p^\prime}}.$$
Since the maximal function $\mathcal{E}$ is strong type $(p,p)$ for $1<p\leq\infty$, hence we get the inequality (\ref{pcase}). Now we will show that for all $F\in L^{p^{\prime}}(\Omega),$ there exists a constant $C$ (independent of $F$) such that
\begin{equation}\label{eq13}
\|F\|_{L^{p^{\prime}}(\Omega)}\leq C\|\mathcal{P}_{z}F\|_{L^{p^{\prime},\infty}(\mathfrak{X})}.
\end{equation}
 Since $\mathcal{P}_{z}F \in L^{p^\prime, \infty}(\x),$ therefore by Lemma \ref{lemmaweak} we have
 $$\frac{1}{N}\sum\limits_{x\in B(0,N)}|\mathcal{P}_{z}F(x)|^{p^{\prime}}\leq C\|\mathcal{P}_{z}F\|^{p^{\prime}}_{L^{p^{\prime},\infty}}.$$

 Let $\omega_0=\{o,\omega_1^0,\ldots, \omega^0_n\ldots\}$ be some fixed element in $\Omega$.
Then the above estimate and the fact $|\varphi_{z}(x)|\asymp q^{-|x|/p^{\prime}}$ for $z=\alpha+i\delta_{p^{\prime}},$ together implies that
 \[\frac{1}{N}\sum\limits_{n=0}^{N}\int\limits_{K}{\left|\frac{\mathcal{P}_{z}F(k\cdot\omega^{0}_{n})}{\varphi_{z}(\omega^{0}_{n})}\right|}^{p^{\prime}}dk\leq C\|\mathcal{P}_{z}F\|_{L^{p^{\prime},\infty}}^{p^{\prime}}\quad\text{for all }N\in\mathbb{N}.\]
 Now the equation (\ref{eq13}) follows from above inequality and \cite[Proposition 1 ]{KP}.

Conversely assume that $u\in \mathbb{E}_z(\x)$. From Theorem \ref{bijection} there exists a martingale $\mathbf F=\{F_{n}\}$ such that
$u(x)=\mathcal{P}_{z}\mathbf F(x)=\lim\limits_{n\rightarrow\infty}\mathcal{P}_{z}F_{n}(x)$.
We may assume that  $F_n\in L^{p^{\prime}}(\Omega)$ for each $n$ and hence from equation (\ref{eq13}),  we have
\begin{equation}
\|F_{n}\|_{L^{p^{\prime}}(\Omega)}\leq C\|\mathcal{P}_{z}F_{n}\|_{L^{p^{\prime},\infty}(\mathfrak{X})}.
\label{eq16}\end{equation}

For each $n,$ we define the operator $\varepsilon_{n}$  by
\[\varepsilon_{n}u(x)=\frac{1}{\#\mathcal{S}(n,x)}\sum\limits_{y\in\mathcal{S}(n,x)}u(y)\quad\text{for all }x\in\mathfrak{X}\]
where
\begin{equation*}
\mathcal{S}(n,x)= \begin{cases}
\vspace*{.2cm} \{x\}&\text{if }|x|\leq n\\
\vspace*{.2cm} \{y\in\mathfrak{X}: |y|=|x|, x_{n}=y_{n}\}&\text{if }|x|>n.\\
\end{cases}
\end{equation*}
It was  proved in \cite{MZ} that for all $n,$ the function $\mathcal{P}_{z}F_{n}$ is given by
\be \mathcal{P}_{z}F_{n}(x)=\varepsilon_{n}(u)(x)\quad\text{for all }x\in\x.\label{eq17}\ee
Let us define $\varepsilon^*u(x)=\sup\limits_n|\varepsilon_{n}u(x)|.$
Assume for a moment that the operator $\varepsilon^*$ is a bounded map from $L^{p^{\prime},\infty}(\x)$ into itself. Then ($\ref{eq16}$), (\ref{eq17}) and the fact that $u\in L^{p^{\prime},\infty}(\mathfrak{X})$ together implies
\[\|F_{n}\|_{L^{p^{\prime}}(\Omega)}\leq C\|u\|_{L^{p^{\prime},\infty} (\x)}\quad\forall n\in\mathbb{N}.\]
This proves that the martingale $\mathbf F$ is given by a $L^{p^{\prime}}(\Omega)$ function say $F$. Hence we have
\[u(x)=\mathcal{P}_{z}F(x).\]
To complete the proof, we will prove the boundedness of the  operator $\varepsilon^*$. Note that $\varepsilon_{n}u(x)=u(x)$ if $|x|\leq n$. Now assume that $|x|=n+k$ for some $k>0$ and let $\{o,x_1,\cdots,x_n,\cdots,x_{n+k}\}$ be the geodesic connecting $o$ to $x.$ Observe that  $\mathcal{S}(n,x)\subset B(x,2k)$  and

$$\#\mathcal{S}(n,x)=q^{k}\asymp \#B(x,2k)^{1/2}.$$
From the above facts we have
$$
|\varepsilon_{n}u(x)|\leq C\frac{1}{\#B(x,2k)^{1/2}}\sum\limits_{y\in B(x,2k)}|u(y)|\leq C \mathcal{M}u(x)
$$
where $\mathcal{M}$ is an operator defined by the formula
\[\mathcal{M}u(x)=\sup_{r\in\mathbb{N}}\frac{1}{\#B(x,r)^{1/2}}\sum\limits_{y\in B(x,r)}|u(y)|.\]
Finally we have
\begin{equation}\label{eq15}
\varepsilon^*u(x)\leq C(\mathcal{M}u(x)+|u(x)|)\quad\text{for all} \quad x\in\mathfrak{\x}.
\end{equation}

Veca proved (see \cite{V}) that $\mathcal{M}$ is a bounded operator from $L^{2,1}(\x)$ to $L^{2,\infty}(\x).$ Hence from ($\ref{eq15}$) we can say that $\varepsilon^*$ is bounded from $L^{2,1}(\x)$ to $L^{2,\infty}(\x).$ It follows from the definition that $\varepsilon^*$ is bounded from $L^{\infty}(\x)$ to $L^{\infty}(\mathfrak{X}).$ Using interpolation results, we conclude that $\varepsilon^*$ is a bounded map from $L^{p^{\prime},\infty}(\x)$ into itself.
This completes the proof.
\end{proof}

\begin{Theorem}
Let $u\in C(\x)$ and $z\in\R\setminus(\tau/2)\Z$. Then  $u(x)=\mathcal{P}_{z}F(x)$ for some $F\in L^{2}(\Omega)$ if and only if
 $u\in L^{2,\infty}(\x)$ and $\mathcal Lu(x)=\gamma(z)u(x)$.  Moreover there exists a positive constant $C_z$ such that for all $F\in L^{2}(\Omega)$
\be
\|\mathcal{P}_{z}F\|_{L^{2,\infty}(\x)}\leq C_z\|F\|_{L^2(\Omega)}.
\label{2case}\ee
\end{Theorem}

\begin{proof}
First we prove that if $F\in L^{2}(\Omega)$ then $\mathcal{P}_{z}F\in L^{2,\infty}(\x)$ and the inequality (\ref{2case}) holds.
We assume that $z=s/2$ where $0<|s|<\tau$.  Let $x\in\x$ and $\{o=x_0,x_1, \ldots ,x_n=x\}$ be the geodesic connecting $o$ to $x$. We expand $\mathcal{P}_{z}F$ in terms of
  $\mathcal{E}_j (F)$ and further replace it by $\sum\limits_{m=0}^{j}\Delta_m (F)$. By adopting the similar approach given in \cite{SP} we have,
\begin{align*}
\mathcal{P}_z F(x)&=\int\limits_{\Omega}p^{(1+is)/2}(x,\omega)F(\omega)d\nu(\omega)
=q^{-|x|(1+is)/2}\sum\limits_{j=0}^{|x|}q^{(1+is)j}\int\limits_{{E_j(x)}\setminus E_{j+1}(x)}F(\omega)d\nu(\omega)\\
&=q^{-|x|(1+is)/2}\left(\mathcal{E}_{0}(F)(\omega)+\sum\limits_{j=1}^{|x|}q^{(1+is)j}\int\limits_{E_j(x)}F(\omega)d\nu(\omega)
-\sum\limits_{j=0}^{|x|-1}q^{(1+is)j}\int\limits_{E_{j+1}(x)}F(\omega)d\nu(\omega)\right)\\%
&=q^{-|x|(1+is)/2}\left(\mathcal{E}_0 (F)(\omega)+\frac{q}{q+1}(1-q^{-1-is})\sum\limits_{j=1}^{|x|}q^{isj}\mathcal{E}_j (F)(\omega)\right)\\
&=q^{-|x|(1+is)/2}\left(\Delta_0 (F)(\omega)+\frac{q}{q+1}(1-q^{-1-is})\sum\limits_{j=1}^{|x|}q^{isj}\sum\limits_{m=0}^{j}\Delta_m (F)(\omega)\right)
\end{align*}
Now changing the order of summation, we get
\begin{align*}
\mathcal{P}_z F(x)&=q^{-|x|(1+is)/2}\Bigg[\Delta_0 (F)(\omega)+\left.\frac{q}{q+1}(1-q^{-1-is})\Bigg(\Delta_0 (F)(\omega)\sum\limits_{j=1}^{|x|}q^{isj}+\sum\limits_{m=1}^{|x|}\Delta_m (F)(\omega)\sum\limits_{j=m}^{|x|}q^{isj}\Bigg)\right]\\
&=q^{-|x|(1+is)/2}\left[\Bigg(1+\frac{1-q^{is|x|}}{1-q^{is}}\frac{q^{1+is}}{q+1}(1-q^{-1-is})\Bigg)\Delta_0 (F)(\omega)\right.\\&\hspace*{5cm}+\left.\frac{q}{q+1}(1-q^{-1-is})\sum\limits_{m=1}^{|x|}q^{ism}\frac{1-q^{is(|x|-m+1)}}{1-q^{is}}\Delta_m (F)(\omega)\right]\\
&=q^{-|x|(1+is)/2}\left[\Bigg(1+\frac{(1-q^{-1-is})}{1-q^{is}}\frac{q^{1+is}}{q+1}\Bigg)\mathcal{E}_0 (F)(\omega)\right.\\&\hspace*{4cm}-\left.\frac{q(1-q^{-1-is})}{(q+1)(1-q^{is})}\left(q^{is(|x|+1)}\mathcal{E}_{|x|} (F)(\omega)-\sum\limits_{m=1}^{|x|}q^{ism}\Delta_m (F)(\omega)\right)\right]\\
\end{align*}
Finally there exists constants $C_1$ and $C_2$ such that
$$|\mathcal{P}_z F(x)|
\leq q^{-|x|/2}\left(C_1\mathcal{E}(|F|)(\omega)+C_2\sup\limits_{|x|\in\N}\left|\sum\limits_{m=1}^{|x|}q^{ism}\Delta_m (F)(\omega)\right|\right)\quad\forall x\in\x, \omega\in E(x).
$$
Thus $|\mathcal{P}_z F(x)|$ is dominated by a sum of two factors $q^{-|x|/2}\mathcal{E}(|F|)$ and $q^{-|x|/2}\mathcal{E}(G)$ respectively, where the function $G$ is defined by  $G(\omega)=\sum\limits_{m=1}^{\infty}q^{ism}\Delta_m (F)(\omega)$. Clearly $G$ is a well defined $L^2$ function and $\|G\|_{L^2(\Omega)}\leq C\|F\|_{L^2(\Omega)}$ (see \cite[Theorem 7, page 95]{ST}).  Since $\mathcal E$ is a strong type $(2,2)$ operator, therefore by using similar argument as in the previous theorem we get the desired inequality (\ref{2case}).

Conversely, assume that $u\in \mathbb{E}_z(\x)$  for some $z=s,~0<|s|<\tau/2$.
By Theorem \ref{bijection} there exists a martingale ${\bf F}=(F_n)$ such that $u(x)=\mathcal{P}_{z}\mathbf F(x)$. Note that
$u(x)=\mathcal{P}_{z}F_{N}(x)$ whenever $|x|\leq N.$ To complete the proof we need to show that $\mathbf F$ is in $L^2(\Omega)$. For this it is enough to show that $\sup\limits_{n}\|F_{n}\|_{L^{2}(\Omega)}<\infty.$ To do so, we will use the given assumption that $u\in L^{2,\infty}(\x)$.
It follows from Lemma \ref{lemmaweak} that for all $N\in \N$ we have
\begin{align}
\frac{1}{N}\sum\limits_{m=0}^{N}q^{m}\int\limits_{K}|\mathcal{P}_{z}F_{N}(k\cdot\omega^{0}_{m})|^{2}dk &\asymp\frac{1}{N}\sum\limits_{m=0}^{N}\sum\limits_{x\in S(0,m)}|\mathcal{P}_{z}F_{N}(x)|^{2}=\frac{1}{N}\sum\limits_{x\in B(0,N)}|\mathcal{P}_{z}F_{N}(x)|^{2}\nonumber\\
&=\frac{1}{N}\sum\limits_{x\in B(0,N)}|u(x)|^{2} \leq C\|u\|^{2}_{L^{2,\infty}},\label{eq18}
\end{align}
where $\omega_{0}=\{o,\omega^{0}_{1},\ldots,\omega^{0}_{n},\ldots\}$ is some fixed element in $\Omega.$ Recall that $F_{N}=\sum\limits_{n=0}^{N}\Delta_{n}F_{N}$ with $\Delta_{n}(\Delta_{n}F_{N})=\Delta_{n}F_{N}$. We need the following formula for $\mathcal{P}_{z}(\Delta_{n}F_{N})$ which directly follows from the calculation given in the \cite[p. 377]{MZ}.
\begin{equation*}
\mathcal{P}_{z}(\Delta_{n}F_{N})(k\cdot\omega_m^0)=\begin{cases}
0&\mbox{if }m<n,\\
B(n,m,z)\Delta_{n}F_{N}(k\cdot\omega_0)&\mbox{if }m\geq n,
\end{cases}
\end{equation*}
where $B(n,m,z)$ is defined by the rule
\begin{equation*}
B(n,m,z)=\begin{cases}
q^{-m/2}(q^{-ism}+{\bf c}(s)(q^{ism}-q^{-ism}))&\mbox{if }n=0,\\
q^{-m/2}{\bf c}(s)q^{is(n-1)}(q^{is(m-n+1)}-q^{-is(m-n+1)})&\mbox{if }n>0.
\end{cases}
\end{equation*}
Hence for every $m$ with $0\leq m\leq N$,
\begin{equation*}
\mathcal{P}_{z}F_{N}(k\cdot\omega_m^0)=\sum\limits_{n=0}^mB(n,m,z)\Delta_nF_N(k\cdot\omega_0).
\end{equation*}
By using above formula and (\ref{ortho}), we have
\begin{align*}
&\frac{1}{N}\sum\limits_{m=0}^{N}q^{m}\int\limits_{K}|\mathcal{P}_{z}F_{N}(k\cdot\omega^{0}_{m})|^{2}dk=\frac{1}{N}
\sum\limits_{m=0}^{N}q^{m}\int\limits_{K}\mathcal{P}_{z}F_{N}(k\cdot\omega^{0}_{m})\overline{\mathcal{P}_{z}F_{N}(k\cdot\omega^{0}_{m})} dk\\
=&\frac{1}{N}\sum\limits_{m=0}^{N}q^{m}\int\limits_{K}\left(\sum\limits_{j=0}^mB(j,m,z)\Delta_jF_N(k\cdot\omega_0)\right)
\left(\overline{\sum\limits_{n=0}^mB(n,m,z)\Delta_nF_N(k\cdot\omega_0)}\right) dk\\
=&\frac{1}{N}\sum\limits_{m=0}^Nq^m\sum\limits_{n=0}^m|B(n,m,z)|^2\|\Delta_nF_N\|_{L^2(\Omega)}^2\\
=&\frac{1}{N}\sum\limits_{n=0}^N\|\Delta_nF_N\|_{L^2(\Omega)}^2\sum\limits_{m=n}^Nq^m|B(n,m,z)|^2.
\end{align*}
From (\ref{eq17}), we have
\begin{equation}\label{eq19}
\sum\limits_{n=0}^N\|\Delta_nF_N\|_{L^2(\Omega)}^2\left(\frac{1}{N}\sum\limits_{m=n}^N|B^\prime(n,m,z)|^2\right)\leq C\|u\|^{2}_{L^{2,\infty}}\quad\text{for all }N\in\N,
\end{equation}
where $B(n,m,z)=q^{-m/2}B^\prime(n,m,z)$. We claim that for all $n\geq 0,$ $\frac{1}{N}\sum\limits_{m=n}^N|B^\prime(n,m,z)|^2\rightarrow2|{\bf c}(s)|^2$ as $N\rightarrow\infty$. The case $n=0$ follows from the fact that ${\bf c}(s)+\overline{{\bf c}(s)}=1$ and the  following explicit expression of $\sum\limits_{m=0}^N|B^\prime(0,m,z)|^2;$
\begin{multline*}
\sum\limits_{m=0}^N|B^\prime(0,m,z)|^2=(N+1)+|{\bf c}(s)|^2\left(2N-q^{2is}\frac{1-q^{2isN}}{1-q^{2is}}-q^{-2is}\frac{1-q^{-2isN}}{1-q^{-2is}}\right)\\+{\bf c}(s)\left(q^{2is}\frac{1-q^{2isN}}{1-q^{2is}}-N\right)+\overline{{\bf c}(s)}\left(q^{-2is}\frac{1-q^{-2isN}}{1-q^{-2is}}-N\right),
\end{multline*}
whereas the case $n>0$ follows from the explicit expression of $\sum\limits_{m=n}^N|B^\prime(n,m,z)|^2$ given below
\begin{multline*}
\sum\limits_{m=n}^N|B^\prime(n,m,z)|^2=|{\bf c}(s)|^2\left(2(N-n+1)-\frac{q^{2is}}{1-q^{2is}}(1-q^{2is(N-n+1)})
\right.\\\left.-\frac{q^{-2is}}{1-q^{-2is}}(1-q^{-2is(N-n+1)})\right).
\end{multline*}
This proves our claim. Hence for any fixed $k\in\mathbb{N}$ there exists a large positive integer (say)  $N_{k}$ greater than $k$  such that
\begin{equation}\label{claim}
\frac{1}{N_{k}}\sum\limits_{m=n}^{N_{k}}|B^\prime(n,m,z)|^2>|{\bf c}(s)|^{2},\mbox{ for all }n\mbox{ with }0\leq n\leq k.
\end{equation}
Since $k< N_k$ therefore by using (\ref{eq19}), (\ref{claim}) and (\ref{ortho}) we have
$$\|F_k\|_{L^2(\Omega)}^2\leq C_z\|u\|_{L^{2,\infty}(\mathfrak{X})}^2.$$

Since $k$ is arbitrary, therefore $\mathbf F$ coincides with some $F\in L^{2}(\Omega)$ and
$ u(x)=\mathcal{P}_{z}F(x).$
This completes the proof.
\end{proof}

\section{Proof of the Theorem A and Theorem B}

\noindent{\bf Proof of Theorem A:}(i) First we assume $p=1$ so that $f\in L^1(\x)$. For $r<\infty$ and $z=\alpha+i\delta_{r^\prime}$, by Minkowski's inequality we have

$$\left(\int\limits_{\Omega}|\tilde{f}(z,\omega)|^{r}d\nu(\omega)\right)^{1/r}\leq\sum\limits_{x\in\x}|f(x)|\left(\int_\Omega p(x,\omega)d\nu(\omega)\right)^{1/r}
\leq \|f\|_{L^1(\x)}.$$

If $r=\infty$ and $z=\alpha+i\delta_1$ then  $\tilde{f}(z,\omega)=\sum\limits_{x\in\x}f(x)p^{i\alpha}(x,\omega)$. Since $|p^{i\alpha}(x,\omega)|\leq1$, hence for $p=1$ the result follows with $C_{p,r}=1$.

Now we discuss the case when $r=p$ or $p^{\prime}$. Our proof is based on the following duality relation and the norm estimates of Poisson transform.
For a finitely supported function $f$ on $\x$ and $F\in C(\Omega)$ we have,
\begin{align}
\int\limits_{\Omega}\tilde{f}(z,\omega)F(\omega)d\nu(\omega)&=\int\limits_{\Omega}\left(\sum\limits_{x\in\mathfrak{X}}f(x)p^{1/2+iz}(x,\omega)\right)F(\omega)d\nu(\omega)\nonumber\\
&=\sum\limits_{x\in\mathfrak{X}}f(x)\left(\int_\Omega p^{1/2+iz}(x,\omega)F(\omega)d\nu(\omega)\right)\nonumber\\
&=\sum\limits_{x\in\mathfrak{X}}f(x)\mathcal{P}_{z}F(x).\label{duality}
\end{align}

Let $r=p$ and $z=\alpha+i\delta_{p^{\prime}},~\alpha\in\R$.
In view of (\ref{duality}), using the duality argument and the estimate (\ref{pcase}) we have

$$\left(\int\limits_{\Omega}|\widetilde{f}(z,\omega)|^p d\nu(\omega)\right)^{1/p}\leq C_{p}\|f\|_{L^{p,1}(\mathfrak{X})}.$$
Next we consider $r=p^{\prime}$ and $z=\alpha+i\delta_{p},~\alpha\in\R$. By the similar duality argument as above, it is enough to show that
\be \|\mathcal{P}_z F\|_{L^{p^{\prime},\infty}(\x)}\leq C_{p}\|F\|_{L^{p}(\Omega)} \;\; \forall F\in L^{p}(\Omega).\label{p'case}\ee
This estimate is already mentioned in \cite{CMS}. The proof of the above estimate is based on the following facts.

\begin{itemize}
\item  For all $x\in\x$, we have
$\mathcal{P}_{i\delta_{p}}F(x)=\mathcal{P}_{i\delta_{p^{\prime}}}{\mathcal{P}_{i\delta_{p^{\prime}}}}^{-1}\mathcal{P}_{i\delta_{p}}F(x)=\mathcal{P}_
{i\delta_{p^{\prime}}}\left(\mathcal{I}_{i\delta_{p^{\prime}}}F\right)(x),$ where $\mathcal{I}_{i\delta_{p^{\prime}}}$ is the intertwining operator.

\item The intertwining operator $\mathcal{I}_{i\delta_{p^{\prime}}}$ is bounded from $L^p(\Omega)$ to $L^{p^{\prime}}(\Omega)$ (For details about intertwining operators, we refer  \cite{MZ}).
\end{itemize}

Finally we will discuss the case  $1<p<r<p^{\prime}<\infty$. If $r<2$ and $z=\alpha+i\delta_{r^{\prime}},~\alpha\in\R$ then by (\ref{pcase}) we have
$$\|\mathcal{P}_z F\|_{L^{r^{\prime},\infty}(\x)}\leq C_{r}\|F\|_{L^{r^{\prime}}(\Omega)} \;\; \forall F\in L^{r^{\prime}}(\Omega).$$

If $r>2$ and $z=\alpha+i\delta_{r^{\prime}},~\alpha\in\R$ then by (\ref{p'case}) we have
$$\|\mathcal{P}_z F\|_{L^{r,\infty}(\x)}\leq C_{r}\|F\|_{L^{r^{\prime}}(\Omega)} \;\; \forall F\in L^{r^{\prime}}(\Omega).$$
Now by interpolation between (\ref{bounded poisson}) and the above two inequalities, we have
$$\|\mathcal{P}_z F\|_{L^{p^{\prime}}(\x)}\leq C_{p,r}\|F\|_{L^{r^{\prime}}(\Omega)}$$
In view of (\ref{duality}) if $z=\alpha+i\delta_{r^{\prime}},~\alpha\in\R$ then
 $$\left(\int\limits_{\Omega}|\widetilde{f}(z,\omega)|^r d\nu(\omega)\right)^{1/r}\leq C_{p,r}\|f\|_{L^{p}(\mathfrak{X})}.$$
This completes the proof.\hfill\qed

Now we prove the $L^2$ case of restriction theorem. Like $p$ case, proof of this theorem is also based on the duality argument and estimate of the Poisson transform.
The key difference in the statement is that $L^2$ restriction theorem does not hold for all $z\in\R$. This can be shown by a simple argument as $\varphi_0$ does not belong
 to $L^{2,\infty}(\x)$. However, $L^p$ restriction holds for all $z$ on the lines $\Im z=\delta_p$ and $\Im z=\delta_{p^\prime}$ if $p\neq 2$.

\noindent{\bf Proof of Theorem B:} Proof of this theorem is a repetition of the argument used in the last theorem. We know from (\ref{duality}) that
for all finitely supported function $f$ on $\x$ and $F\in C(\Omega)$
$$\int\limits_{\Omega}\tilde{f}(z,\omega)F(\omega)d\nu(\omega)=\sum\limits_{x\in\mathfrak{X}}f(x)\mathcal{P}_{z}F(x).$$ Now the result will follow from above equation and (\ref{2case}). This completes the proof.\hfill\qed


 
\end{document}